\documentclass[12pt]{amsart}

\usepackage{multirow}
\usepackage{tikz}
\usepackage{mathrsfs}
\usepackage{color}
\usepackage{tikz}
\usepackage[all,cmtip]{xy}
\usepackage{amsmath,amsfonts,amsthm,amssymb}
\usepackage{bbm}
\usepackage{graphics}
\usepackage{graphicx} 
\usepackage{mathtools}
\usepackage{geometry}
\numberwithin{equation}{section}

\theoremstyle{plain}

\newtheorem{theorem}{Theorem}[section]

\newtheorem{lemma}[theorem]{Lemma}
\newtheorem{proposition}[theorem]{Proposition}
\newtheorem{corollary}[theorem]{Corollary}
\newtheorem{definition}[theorem]{Definition}
\newtheorem{example}[theorem]{Example}
\newtheorem{remark}[theorem]{Remark}

\newtheorem{notation}{Notation}

\newcommand{\depth}{\operatorname{depth}}

\newcommand{\Hom}{\operatorname{Hom}\nolimits}

\newcommand{\id}{\operatorname{id}}

\newcommand{\rad}{\operatorname{rad}\nolimits}

\renewcommand{\mod}{\mbox{-}\mathrm{mod}}
\newcommand{\Modp}{\mbox{-}\mathrm{Mod}_{\mathcal{P}}}
\newcommand{\modp}{\mbox{-}\mathrm{mod}_{\mathcal{P}}}
\newcommand{\Modfpd}{\mbox{-}\mathrm{Mod}_{\rm fpd}}
\newcommand{\modfpd}{\mbox{-}\mathrm{mod}_{\rm fpd}}

\newcommand{\smod}{\mbox{-}\underline{\mathrm{mod}}}
\newcommand{\sMod}{\mbox{-}\underline{\mathrm{Mod}}}
\newcommand{\Mod}{\mbox{-}\mathrm{Mod}}
\newcommand{\Gproj}{\mbox{-}\mathrm{Gproj}}
\newcommand{\GProj}{\mbox{-}\mathrm{GProj}}
\newcommand{\SGProj}{\mbox{-}\mathrm{SGProj}}
\newcommand{\SGproj}{\mbox{-}\mathrm{SGproj}}
\newcommand{\sGproj}{\mbox{-}\underline{\mathrm{Gproj}}}
\newcommand{\sGProj}{\mbox{-}\underline{\mathrm{GProj}}}
\newcommand{\proj}{\mbox{-}\mathrm{proj}}
\newcommand{\Proj}{\mbox{-}\mathrm{Proj}}
\newcommand{\Fpd}{\mbox{-}\mathrm{Fpd}}
\newcommand{\fpd}{\mbox{-}\mathrm{fpd}}

\newcommand{\Ext}{\operatorname{Ext}\nolimits}

\newcommand{\Gpd}{\operatorname{Gpd}}
\newcommand{\Gdim}{\operatorname{G\mbox{-}dim}}
\newcommand{\Pdim}{\operatorname{per.dim}}
\newcommand{\Findim}{\operatorname{Fin.dim}}
\newcommand{\findim}{\operatorname{fin.dim}}

\newcommand{\gldim}{\operatorname{gl.dim}}

\newcommand{\injdim}{\operatorname{inj.dim}}
\newcommand{\projdim}{\operatorname{proj.dim}}

\renewcommand{\L}{\Lambda}
\newcommand{\Lop}{\Lambda^{\rm op}}
\newcommand{\G}{\Gamma}
\newcommand{\Le}{\Lambda^\textrm{e}}
\newcommand{\Ge}{\Gamma^\textrm{e}}

\begin{document}
%
%
\title[Periodic dimensions and...]{Periodic dimensions and some homological properties of eventually periodic algebras}

\author[S. Usui]{Satoshi Usui}
\address{%
Graduate School of Science, Department of Mathematics,
Tokyo University of Science, 
1-3, Kagurazaka, Shinjuku-ku, 
Tokyo, 162-8601,
Japan}
\address{%
Present address: Department of General Education, Tokyo Metropolitan College of Industrial Technology, 
8-17-1, Minamisenju, Arakawa-ku, 
Tokyo, 116-8523,
Japan}

\email{usui@metro-cit.ac.jp} 


\thanks{}


\begin{abstract}
For an eventually periodic module, we have the degree and the period of its first periodic syzygy. 
This paper studies the former under the name \lq\lq periodic dimension\rq\rq. 
We give a bound for the periodic dimension of an eventually periodic module with finite Gorenstein projective dimension. 
We also provide a method of computing the Gorenstein projective dimension of an eventually periodic module under certain conditions. 
Besides, motivated by recent results of Dotsenko, G{\' e}linas and Tamaroff and of the author, we determine the bimodule periodic dimension of an eventually periodic Gorenstein algebra.   
Another aim of this paper is to obtain some of the basic homological properties of eventually periodic algebras. 
We show that a lot of homological conjectures hold for this class of algebras. 
As an application, we characterize eventually periodic Gorenstein algebras in terms of bimodules Gorenstein projective dimensions. 
\end{abstract}


\maketitle
\tableofcontents

\section{Introduction} \label{Introduction}

Throughout this paper, we assume that all rings are associative and unital, and let $k$ denote a field. Moreover, by a module, we mean a left module unless otherwise stated.

Let $R$ be a left Noetherian semiperfect ring. Then any finitely generated $R$-module $M$ admits a minimal projective resolution $P_{\bullet} \rightarrow M$ with each $P_i$ finitely generated projective.
Such a resolution is uniquely determined up to isomorphism in the category of (chain) complexes over $R$.
We call the kernel of the $(n-1)$-th differential  $d_{n-1} : P_{n-1} \rightarrow P_{n-2}$ of the minimal projective resolution $P_{\bullet} \rightarrow M$ the {\it $n$-th syzygy} of $M$ and denote it by $\Omega_{R}^{n}(M)$.
The $R$-module $M$ is called {\it periodic} if  $\Omega_R^{p}(M)$ is isomorphic to $M$ as $R$-modules for some $p>0$. 
The least such $p$ is called the {\it period} of $M$.
We say that $M$ is {\it eventually periodic} if $\Omega_R^{n}(M)$ is periodic for some $n \geq 0$.

So far, various characterizations of eventually periodic modules have been provided in \cite{Avramov_1989,Bergh_2006,Croll_2013,Eisenbud_1980,Gasharov-Peeva_1990,kupper2010two,Usui_2022}. 
At the same time, the degree $n$ and the period $p$ of the first periodic syzygy $\Omega_R^{n}(M)$ of an eventually periodic $R$-module $M$ have been also examined:
for instance, in case $R$ is a commutative Noetherian local ring and $M$ has finite virtual projective dimension, Avramov \cite[Theorem 4.4]{Avramov_1989} obtained that the degree $n$ is bounded above by $\depth R -\depth M +1$, and the period $p$ is either $1$ or $2$, where $\depth N$ denotes the depth of an $R$-module $N$.
We point out that the degree $n$ is bounded below by $\depth R -\depth M$, which can be observed by using \cite[Lemma 1.2.6]{Avramov_1998}. 
This observation is valid for any eventually periodic modules over a commutative Noetherian local ring.
Also, we remark that the above result of Avramov recovers a well-known result of Eisenbud \cite[Theorem 4.1]{Eisenbud_1980} over complete intersections.
On the other hand, the author \cite[Theorem 2.4]{Usui_2022} showed that the Tate cohomology ring of the eventually periodic $R$-module $M$ has an invertible homogeneous element whose degree is the period $p$ of some periodic syzygy $\Omega_R^{n}(M)$ (without any additional assumption on $R$ and $M$).
Furthermore, there are results on the values of $n$ and $p$: for example, \cite[Theorem 2.5]{Bergh_2006}, \cite[Theorem 1.2]{Gasharov-Peeva_1990}, \cite[the proof of Corollary 6.4]{Dotsenko-Gelinas-Tamaroff_2023} and \cite[Proposition 4.2]{Usui21_No02}.

In this paper, we are concerned with the degrees of the first periodic syzygies of eventually periodic modules.
Since any module over a left perfect ring admits by definition a minimal projective resolution, we work with such rings $R$ and define eventually periodic $R$-modules by dropping the property \lq\lq finitely generated\rq\rq.  
We then introduce the notion of the {\it periodic dimension} of an $R$-module $M$. 
It is defined to be the infimum of the degrees $n$ of the periodic syzygies $\Omega_R^{n}(M)$ of $M$.
From the definition, a module being eventually periodic is equivalent to its periodic dimension being finite.
In this case, the value of the periodic dimension equals the degree of the first periodic syzygy. 
Moreover, it also follows that a module is periodic if and only if the module is of periodic dimension $0$, which justifies the terminology \lq\lq periodic dimension\rq\rq.

After providing some basic properties of periodic dimensions, we use the notion of {\it Gorenstein projective dimensions}, introduced by Holm \cite{HenrikHolm04}, to study periodic dimensions. 
As the first main result of this paper, we show that the periodic dimension of an eventually periodic module $M$ of finite Gorenstein projective dimension $r$ is either $r$ or $r+1$ (Theorem \ref{claim_6}).
Moreover, under the additional assumption that $M$ is finitely generated over a left artin ring, we will not only decide when the former case occurs but also describe how to obtain the Gorenstein projective dimension from the periodic dimension; see Corollary \ref{claim_49}.
Note that the corollary enables us to compute the Gorenstein projective dimensions of arbitrary finitely generated modules over CM-finite Gorenstein algebras; see Example \ref{example_5}.
Also, in the case of (both left and right) Noetherian semiperfect rings, we give an analogous result to the above main
result (Theorem \ref{claim_50}).

We also consider the bimodule periodic dimension of a finite dimensional eventually periodic algebra (i.e.~a finite dimensional algebra that is eventually periodic as a regular bimodule).
In the proof of \cite[Proposition 2.4]{Dotsenko-Gelinas-Tamaroff_2023}, Dotsenko, G{\' e}linas, and Tamaroff showed that finite dimensional monomial Gorenstein algebras $\L$ are eventually periodic, and their bimodule periodic dimensions are bounded above by $\injdim_{\L}\L+1$, where $\injdim_{\L}\L$ stands for the injective dimension of $\L$ over $\L$.
Also, the author \cite[Proposition 4.2]{Usui21_No02} obtained that the bimodule periodic dimensions of finite dimensional eventually periodic Gorenstein algebras $\L$ are bounded below by $\injdim_{\L}\L$. 
These two results are the motivation for the second main result of this paper that the bimodule periodic dimension of an eventually periodic Gorenstein algebra $\L$ is either $\injdim_{\L}\L$ or $\injdim_{\L}\L+1$ (Theorem \ref{claim_10}).
We remark that our second main result is a consequence of the first one.

This paper also focuses on a homological aspect of finite dimensional eventually periodic algebras.
It turns out that many homological conjectures such as the periodicity conjecture, the finitistic dimension conjecture, the Gorenstein symmetric conjecture, and the Auslander conjecture hold for this class of algebras (Propositions \ref{claim_3}, \ref{claim_18}, \ref{claim_19} and \ref{claim_28}).
As its application, we obtain the third main result of this paper that a finite dimensional eventually periodic algebra is Gorenstein if and only if its bimodule Gorenstein projective dimension is finite (Theorem \ref{claim_38}).
This illustrates that our second main result is the best possible in the sense that the algebras for which we can describe the bimodule periodic dimensions as a consequence of the first main result are precisely eventually periodic Gorenstein algebras (cf.~Remark \ref{remark_1}).

This paper is organized as follows.
In Section \ref{Preliminaries}, we recall some basic facts that are used in this paper.
In Section \ref{Periodic dimension}, we define and study the periodic dimensions of modules over left perfect rings.
In Section 4, we concentrate on the bimodule periodic dimensions of finite dimensional eventually periodic algebras.
In Section 5, we continue to work with finite dimensional eventually periodic algebras and investigate them from a homological point of view.

\subsection*{Conventions and notation}
For any ring $R$ and any $R$-module $M$, we denote by $R\Mod$ (resp.\ $R\mod$) the category of (resp.\ finitely generated) $R$-modules, by $\gldim R$ the global dimension of $R$, and by $\projdim_{R}M$ the projective dimension of $M$.
We define four full subcategories of $R\Mod$ as follows:
\begin{align*}
    &R\Proj := \left\{\, M \in R\Mod \mid  M \mbox{ is projective} \,\right\}; \\
    &R\Fpd := \left\{\, M \in R\Mod \mid \projdim_R M < \infty \,\right\}; \\
    &R\Modp := \left\{\, M \in R\Mod \mid M \mbox{ has no non-zero direct summand in $R\Proj$} \,\right\}; \\
    &R\Modfpd := \left\{\, M \in R\Mod \mid M \mbox{ has no non-zero direct summand in $R\Fpd$} \,\right\}.
\end{align*}
Similarly, the four full subcategories $R\proj,$ $R\fpd,$ $R\modp$ and $R\modfpd$ of $R\mod$ are defined.
For a collection $\mathcal{X}$ of $R$-modules, we denote by ${}^{\perp}\mathcal{X}$ the full subcategory of $R\Mod$ given by 
\begin{align*} 
   {}^{\perp}\mathcal{X} := \left\{\, M \in R\Mod \mid  \Ext_{R}^{i}(M, X) = 0 \mbox{ for all } i > 0  \mbox{ and all } X \in \mathcal{X} \,\right\}.
\end{align*} 
Finally, for any $i \in \mathbb{Z}$ and any complex $X_\bullet$, we denote by $\Omega_{i}(X_\bullet)$ the cokernel of the $(i+1)$-th  differential $d_{i+1}: X_{i+1} \rightarrow X_{i}$.


\section{Preliminaries} \label{Preliminaries} 

In this section, we recall some basic facts related to Gorenstein projective modules, Gorenstein projective dimensions, and Gorenstein rings.

\subsection{Gorenstein projective modules} 
Let $R$ be a ring.
Recall that an acyclic complex $T_\bullet$ of projective $R$-modules is {\it totally acyclic} if $\Hom_{R}(T_\bullet, Q)$ is acyclic for any $Q \in R\Proj$.
An $R$-module $M$ is called {\it Gorenstein projective} \cite{Enochs-Jenda_1995_MathZ} if there exists a totally acyclic complex $T_\bullet$ such that $\Omega_{0}(T_\bullet) \cong M$ in $R\Mod$. 
For example, projective modules are Gorenstein projective.
As will be seen in the next subsection, if $R$ is a Noetherian ring, then finitely generated Gorenstein projective $R$-modules are precisely {\it totally reflexive $R$-modules} \cite{Avramov-Martsinkovsky_2002}.

Let $n$ be a positive integer. 
Following \cite[Definition 2.1]{Bennis-Mahdou_2009}, we say that the $R$-module $M$ is {\it $n$-strongly Gorenstein projective} if there exists an exact sequence of $R$-modules
\begin{align*}
     0 \rightarrow M \rightarrow P_{n-1} \rightarrow \cdots \rightarrow P_{0} \rightarrow M \rightarrow 0
\end{align*}
with each $P_{i}$ projective such that $\Hom_{R}(-, Q)$ leaves the sequence exact for any $Q \in R\Proj$.
Recall that the {\it stable category} $R\sMod$ of $R\Mod$ is the category whose objects are the same as $R\Mod$ and morphisms are given by $\underline{\Hom}_{\L}(M, N) := \Hom_{\L}(M, N)/\mathcal{P}(M, N)$,
where $\mathcal{P}(M, N)$ is the group of morphisms from $M$ to $N$ factoring through a projective module.
We then observe that $M$ is $n$-strongly Gorenstein projective if and only if $\Omega_{n}(P_\bullet) \cong M$ in $R\sMod$ for some (hence for any) projective resolution $P_\bullet \rightarrow M$ of $M$, and $\Ext_{R}^{i}(M, P) = 0$ for all $i$ with  $1 \leq i \leq n$ and all $P \in R\Proj$ (cf.~\cite[Proposition 2.2.17]{X-WChen_2017}).

The category $R\GProj$ of Gorenstein projective $R$-modules is a Frobenius category whose projective objects are precisely projective $R$-modules, so that the stable category $R\sGProj$ of $R\GProj$ carries a structure of a triangulated category (cf.~\cite[Proposition 2.1.11]{X-WChen_2017}).
If $\Sigma$ denotes the shift functor on $R\sGProj$, then any totally acyclic complex  $T_\bullet$ associated with a Gorenstein projective $R$-module $M$ has the property that $\Sigma^{i}M = \Omega_{-i}(T_\bullet)$ for all $i \in \mathbb{Z}$.
Moreover, we know by \cite[Lemma 2.3.4]{Veliche_2006} that $\Sigma^{-i}M = \Omega_{i}(P_\bullet)$ for any $i  \geq 0$ and any projective resolution $P_\bullet \rightarrow M$ of $M$.
On the other hand, one observes that $R\GProj \subseteq{{}^{\perp}(R\Proj)} = {}^{\perp}(R\Fpd)$.

We denote by $n\mbox{-}R\SGProj$ the category of $n$-strongly Gorenstein projective $R$-modules.
It follows from \cite[Proposition 2.5]{Bennis-Mahdou_2009} that the following inclusions hold for each $n>0$:
\begin{align*}
    R\Proj \subseteq{1\mbox{-}R\SGProj}  \subseteq{n\mbox{-}R\SGProj} \subseteq{R\GProj}.
\end{align*}

In case $R$ is a Noetherian ring, one can deduce analogous results as in the above for $R\Gproj$  and $n\mbox{-}R\SGproj$, where $R\Gproj$ (resp.~$n\mbox{-}R\SGproj$) stands for the category of finitely generated Gorenstein projective (resp.~$n$-strongly Gorenstein projective) $R$-modules.

\subsection{Gorenstein projective dimensions} \label{GorensteinProjectiveDimensions}
Let $R$ be a ring.
Following \cite[Definition 2.8]{HenrikHolm04}, we define the {\it Gorenstein projective dimension} $\Gpd_{R}M$ of an $R$-module $M$ by the infimum of the length $n$ of an exact sequence of $R$-modules 
\begin{align*}
0 \rightarrow G_n \rightarrow \cdots \rightarrow G_1 \rightarrow G_0 \rightarrow M \rightarrow 0 
\end{align*}
with each $G_i$ Gorenstein projective. 
From the definition, there is an inequality 
\begin{align*}
    \Gpd_R M \leq \projdim_R M.
\end{align*}
We know from  \cite[Proposition 2.27]{HenrikHolm04} that the equality holds if $M$ has finite projective dimension.
Moreover, it was proved in \cite[Theorem 2.20]{HenrikHolm04} that if $M$ has finite Gorenstein projective dimension, then we have
\begin{align} \label{eq_2}
    \Gpd_{R}M = \sup \left\{ i \geq 0 \mid \Ext_{R}^{i}(M, Q) \not= 0 \mbox{ for some } Q \in R\Proj \right\}.
\end{align}

Now, suppose that $R$ is a Noetherian ring.
Recall from \cite{Avramov-Martsinkovsky_2002} that the {\it Gorenstein dimension} $\Gdim_{R}M$ of a finitely generated $R$-module $M$ is defined to be the infimum of the length $n$ of an exact sequence of finitely generated $R$-modules \[ 0 \rightarrow X_n \rightarrow \cdots \rightarrow X_1 \rightarrow X_0 \rightarrow M \rightarrow 0 \] with each $X_i$ totally reflexive. 
It was observed in \cite[2.4.1]{Veliche_2006} that  
\begin{align} \label{eq_11}
    \Gdim_{R}M = \Gpd_{R}M
\end{align}
for any $M\in R\mod$.

\subsection{Gorenstein rings}  \label{Preliminaries_GorensteinRings}

A Noetherian ring $R$ is called {\it Gorenstein} (or {\it Iwanaga-Gorenstein})  if $R$ has finite injective dimension as a left and as a right $R$-module (cf.~\cite{Iwanaga_1980,Buchweitz_1986}).
It follows from \cite[Lemma A]{Zaks69} that any Gorenstein ring $R$ satisfies $\injdim_{R}R = \injdim_{R^{\rm op}}R$.
We hence call a Gorenstein ring $R$
with $\injdim_{R}R = d$ a {\it $d$-Gorenstein} ring.
$0$-Gorenstein rings are just self-injective rings.
Also, finitely generated Gorenstein projective modules over a Gorenstein ring are sometimes called {\it maximal Cohen-Macaulay modules} \cite{Buchweitz_1986}. 
The following two results which will be used in this paper are due to Veliche \cite[2.4.2]{Veliche_2006} and Dotsenko, G\'{e}linas and Tamaroff \cite[Proposition 2.4]{Dotsenko-Gelinas-Tamaroff_2023}.

\begin{proposition}[Veliche] \label{Veliche_2006_2.4.2}
Let $R$ be a Noetherian ring and $d$ a non-negative integer.
Then the following conditions are equivalent.
\begin{enumerate}
    \item $\injdim_{R}R \leq d$ and $\injdim_{R^{\rm op}}R \leq d$.
    \item $\Gpd_{R} M \leq d$ for any $R$-module $M$.
    \item $\Gdim_{R} M \leq d$ for any finitely generated $R$-module $M$.
\end{enumerate}
\end{proposition}

\begin{proposition}[Dotsenko-Gélinas-Tamaroff]
\label{Dotsenko-Gelinas-Tamaroff_2023_Proposition 2.4}
Let $\L$ be an artin algebra with Jacobson radical $J(\L)$.
Then $\L$ is a Gorenstein algebra if and only if $\L/J(\L)$ has finite Gorenstein dimension as a $\L$-module. 
In this case, we have $\Gdim_{\L}\L/J(\L) = \injdim_{\L}\L$.
\end{proposition}

It follows from Proposition \ref{Veliche_2006_2.4.2} that for a self-injective ring $R$, we have  
\begin{align*}
    R\GProj = R\Mod \mbox{\quad  and   \quad} R\Gproj = R\mod.
\end{align*}

Let $\L$ be a finite dimensional $d$-Gorenstein algebra (over the field $k$).
Then \cite[Lemma 6.1]{benson_iyengar_krause_pevtsova_2020} implies that the enveloping algebra $\Le := \L \otimes_k \Lop$ of $\L$ is a finite dimensional $(2d)$-Gorenstein algebra.
We note that for any finite dimensional algebra $\L$, $\Le$-modules can be identified with $\L$-bimodules on which the ground field $k$ acts centrally.


\section{Periodic dimensions} \label{Periodic dimension}
In this section, we introduce the notion of the periodic dimension of a module and present some of its basic properties. 
Moreover, we inspect the periodic dimension of an eventually periodic module having finite Gorenstein projective dimension.
Throughout this section, let $R$ be a left perfect ring unless otherwise stated.
We start with a quick review of syzygies.

Recall that the {\it syzygy} $\Omega_{R}(M)$ of an $R$-module $M$ is the kernel of a projective cover $P \rightarrow M$ with $P\in R\Proj$.
It is known that $\Omega_{R}(M)$ is uniquely determined up to isomorphism.
For $n \geq 0$, we set $\Omega_{R}^{n}(M) := \Omega_{R}(\Omega_{R}^{n-1}(M))$, called the {\it $n$-th syzygy} of $M$, where it is understood that $\Omega_{R}^{0}(M) := M$.
Observe that for any family $\{ M_i \}_{i \in I}$ of $R$-modules, there exists an isomorphism in $R\Mod$
\begin{align} \label{eq_9}
    \Omega_R\left(\bigoplus_{i\in I} M_i\right) \cong \bigoplus_{i\in I} \Omega_R(M_i).
\end{align}
On the other hand, there exists a well-defined functor $\Omega_{R} : R\sMod\rightarrow R\sMod$, called the {\it syzygy functor}, sending a module $M$ to its syzygy $\Omega_{R}(M)$.
In case $R$ is furthermore left Noetherian, the syzygy functor  $\Omega_R$ restricts to an endofunctor on the stable category $R\smod$ of $R\mod$.

We now define eventually periodic modules.

\begin{definition} \label{def_1}
{\rm 
An $R$-module $M$ is said to be {\it periodic} if there exists an integer $p>0$ such that $\Omega_{R}^{p}(M) \cong M$ in $R\Mod$.  
The smallest $p>0$ with this property is called the {\it period} of $M$.
We call $M$ {\it eventually periodic} if there exists an integer $n \geq 0$ such that  $\Omega_{R}^{n}(M)$ is periodic.
}
\end{definition}

In this paper, an {\it $(n, p)$-eventually periodic module} means an eventually periodic module whose $n$-th syzygy is the first periodic syzygy of period $p$. If $n=0$, such an eventually periodic module is called {\it $p$-periodic}.
For example, modules of finite projective dimension $n$ are $(n+1, 1)$-eventually periodic.
We now provide an example of $(n, p)$-eventually periodic modules.

\begin{example} \label{example_3}
{\rm 
Fix two integers $n \geq 0$ and $p>0$, and consider the truncated algebra $\L = kQ/R^2$, where $Q$ is the following quiver:
\vspace{5.5mm}
\begin{align*}
\xymatrix{
n \ar[r] & n-1 \ar[r] & \cdots \ar[r] & 1 \ar[r] & 0 \ar[r] & -1 \ar[r] & \cdots \ar[r] & -p+1 \ar@/_20pt/[lll]}
\end{align*}
and $R$ is the arrow ideal of the path algebra $k Q$.
We denote by $S_i$ the simple $\L$-module  associated with the vertex $i$.
A direct calculation shows that  $S_i$ is $(i, p)$-eventually periodic if  $1 \leq i \leq n$ and is $p$-periodic if $-p+1 \leq i \leq 0$.
In particular, $S_n$ is $(n, p)$-eventually periodic.
}\end{example}

It is easy to see that if $M$ is a periodic module, then all its syzygies are periodic and have the same period as $M$.
This implies that the class of periodic modules is closed under syzygies. 
Therefore, it is natural to introduce the following notion.

\begin{definition} \label{def_2} 
{\rm 
The {\it periodic dimension} of an $R$-module $M$ is defined by
\[
\Pdim_{R}M := \inf\left\{\,n \geq 0 \mid  \Omega_{R}^{n}(M) \mbox{\rm \ is periodic}\,\right\}.
\]
}
\end{definition}

By definition, $M$ is eventually periodic if and only if $\Pdim_{R}M < \infty$. 
In this case, $\Pdim_R M$ is equal to the degree $n$ of the first periodic syzygy $\Omega_R^{n}(M)$.
For instance, if $M$ is an $R$-module of finite projective dimension, then $\Pdim_R M = \projdim_R M +1$.
Also, if $M$ has finite periodic dimension $n$, then we have  
    \begin{align*}
        \Pdim_R \Omega_R^{i}(M) = 
        \begin{cases}
        n-i & \mbox{ if $0 \leq i \leq n$,} \\
        0 & \mbox{  if $i > n$ .}
        \end{cases}
    \end{align*}
Moreover, for any family $\{ M_i \}_{i \in I}$ of $R$-modules, the isomorphism (\ref{eq_9}) yields an inequality
    \begin{align} \label{eq_7}
    \Pdim_R \bigoplus_{i\in I} M_i \leq \sup\{\, \Pdim_R M_i \mid i \in I\,\}.
    \end{align}
As in the following example, the equality does not hold in general.

\begin{example} \label{example_2}
{\rm 
Let $Q$ be the following quiver:
\begin{align*}
\xymatrix{
5 \ar@<0.6ex>[r] & 4 \ar@<0.6ex>[l] \ar[r] & 3 \ar[r] & 2 \ar[r] & 1 \ar[r] & 0 }
\end{align*}
and consider $\L = kQ/R^2$.
A direct calculation shows that
\begin{align*}
    \projdim_{\L}S_i = 
    \begin{cases}
    i & 
    \mbox{ if $ 0 \leq i \leq 3$, }\\
    \infty & \mbox{ if $4 \leq i \leq 5$}
    \end{cases}
\end{align*}
and 
\begin{align*}
    \Pdim_{\L}S_i = 
    \begin{cases}
    i+1 & 
    \mbox{ if $ 0 \leq i \leq 3$, }\\
    i-1 & \mbox{ if $4 \leq i \leq 5$,}
    \end{cases}
\end{align*}
and that $\Omega_{\L}^{3}(S_4) = S_1 \oplus S_3 \oplus S_5$.
We then have that  
\begin{align*}
    \Pdim_\L \Omega_\L^{3}(S_4) = 0 < 4 = \max\{\, \Pdim_\L S_i \mid i = 1, 3, 5 \,\}.
\end{align*}
}\end{example}

Example \ref{example_2} concludes that the class of periodic modules is not closed under direct summands in general. The following observation says that direct summands of finitely generated periodic modules are at least eventually periodic and moreover decides when such a direct summand is again periodic.

\begin{proposition}  \label{claim_46}
Let $R$ be a left artin ring and $M$ a finitely generated periodic $R$-module. Then the following statements hold.
\begin{enumerate}
    \item Any indecomposable direct summand of $M$  is eventually periodic.
    \item Every indecomposable direct summand of $M$ is periodic if and only if $M$ has no non-zero direct summand with finite projective dimension.    
\end{enumerate}
\end{proposition}
\begin{proof}
Suppose that $M$ is $p$-periodic and that 
\begin{align} \label{eq_8}
    M = L_1  \oplus \cdots \oplus L_r \oplus N_1 \oplus  \cdots \oplus N_s \oplus N_{s+1} \oplus \cdots \oplus N_t 
\end{align}
is a decomposition of indecomposable $R$-modules such that 
\begin{enumerate}
    \item[(i)] $\projdim_{R} L_i = \infty$ for $1 \leq i \leq r$;
    \item[(ii)] $p \leq \projdim_{R} N_i < \infty$ for $1 \leq i \leq s$; and
    \item[(iii)] $\projdim_{R} N_i < p$ for $s+1 \leq i \leq t$.
\end{enumerate}

For (1), it is enough to show that each $L_i$ is eventually periodic. 
Since $\Omega_R^p(M) \cong M$, we have the following isomorphism in $R\mod$
\begin{align*}
&\Omega_R^{p}(L_1)  \oplus \cdots \oplus  \Omega_R^{p}(L_r) \oplus  \Omega_R^{p}(N_1) \oplus  \cdots \oplus  \Omega_R^{p}(N_s) \\[2mm]
 \cong &\   
L_1  \oplus \cdots \oplus L_r \oplus N_1 \oplus   \cdots \oplus N_t.
\end{align*}
Since $\projdim_R \Omega_R^{p}(L_i) = \infty$ and $\projdim_R \Omega_R^{p}(N_i) < \infty$, the Krull-Schmidt theorem implies that there exists a bijection $\sigma : \{1, \ldots, r\} \rightarrow \{1, \ldots, r\}$ such that  
\begin{align*} 
    \Omega_R^{p}(L_i) \cong L_{\sigma(i)} \oplus N_i^\prime
\end{align*}
in $R\mod$ for each $i$, where $N_i^\prime := \bigoplus_{j \in I(i)}N_j$ for some index set $I(i) \subseteq{\{1, \ldots, t\}}$.
Applying $\Omega_R^{lp}$ with $l:=r!$ to the above isomorphism, we obtain the following isomorphisms in $R\mod$:
\begin{align*} 
    \Omega_R^{(l+1)p}(L_i) 
    &\cong \Omega_R^{lp}\!\left(L_{\sigma(i)}\right) \oplus \Omega_R^{lp}\!\left(N_i^\prime\right) \nonumber
    \\[1.5mm]
    &\cong \Omega_R^{(l-1)p}\!\left(L_{\sigma^2(i)}\right) \oplus \Omega_R^{(l-1)p}\!\left(N_{\sigma(i)}^\prime\right)\oplus \Omega_R^{lp}\!\left(N_i^\prime\right) \nonumber
    \\
    &\ \,\vdots\nonumber
    \\
    &\cong L_{\sigma^{l+1}(i)} \oplus N_{\sigma^l(i)}^\prime \oplus \left( \bigoplus_{j=1}^{l} \Omega_R^{j p}\!\left(N_{\sigma^{l-j}(i)}^\prime\right) \right) \nonumber
    \\[1.5mm]
    &\cong \Omega_R^{p}(L_i) \oplus \left(\bigoplus_{j=1}^{l} \Omega_R^{j p}\!\left(N_{\sigma^{l-j}(i)}^\prime \right) \right). 
\end{align*}
Since the direct summand 
\begin{align*}
   \bigoplus_{j=1}^{l} \Omega_R^{j p}\!\left(N_{\sigma^{l-j}(i)}^\prime \right)
\end{align*}
has finite projective dimension, say $d_i$, 
we deduce that
\begin{align*}
    \Omega_R^{(p+d_i)+lp}(L_i) =  \Omega_R^{(l+1)p+d_i}(L_i) \cong  \Omega_R^{p+d_i}(L_i)
\end{align*}
in $R\mod$. 
This means that the periodic dimension of $L_i$ is finite and at most $p + d_i$.

For (2), it suffices to show the\ \lq\lq if\ \!\rq\rq\ part. 
When $M$ is in $R\modfpd$, or equivalently, $t=0$ in the decomposition (\ref{eq_8}), one gets a bijection $\sigma : \{1, \ldots, r\} \rightarrow \{1, \ldots, r\}$ such that $\Omega_R^{p}(L_i) \cong L_{\sigma(i)}$ for each $i$.
It then follows that $\Omega_R^{l p}(L_i) \cong L_{\sigma^{l}(i)} = L_{i}$.
\end{proof}

We have the following consequence of Proposition \ref{claim_46}.

\begin{corollary} \label{claim_47}
Let $\{ M_i \}_{i \in I}$ be a finite set of finitely generated modules over a left artin ring $R$. 
Assume that $M := \bigoplus_{i\in I} M_i$ is $(n, p)$-eventually periodic with $\Omega_R^{n}(M)$ in $R\modfpd$. Then $\Omega_R^{n}(M_i)$ is periodic for all $i \in I$, and we have
    \begin{align*}
        \Pdim_R M = \max\left\{\, \Pdim_R M_i \mid i \in I\,\right\}.
    \end{align*}
\end{corollary}  
\begin{proof} 
We know from Proposition \ref{claim_46} (2) that each indecomposable direct summand of $\Omega_R^{n}(M)$ is periodic.
Since a direct sum of periodic modules is periodic as well, we see that each $\Omega_R^{n}(M_i)$ is periodic. 
This implies that $\Pdim_R M_i \leq \Pdim_R M$ for each $i \in I$, which completes the proof.
\end{proof}

Let $M$ be an $R$-module and $n$ a positive integer. 
One easily observes that if $\Ext_{R}^{n}(M, X)= 0$ for all $X \in R\Fpd$, then $\Omega_R^{n}(M)$ belongs to $R\Modfpd$.

Next, we treat eventually periodic modules of finite Gorenstein projective dimension. 
We begin with the following lemma.

\begin{lemma}  \label{claim_1}
Let $M$ be an $R$-module such that $\Omega_{R}^{n+p}(M)$ $\cong$ $\Omega_{R}^{n}(M)$ in $R\sMod$ for some $n \geq 0$ and $p>0$.
Then we have $\Pdim_{R}M \leq n+1$.  
Moreover, the period of the first periodic syzygy of $M$ divides $p$.
\end{lemma}
\begin{proof} 
By \cite[Proposition 1.44]{Auslander-Bridger_1969}, there exist two projective $R$-modules $P$ and $Q$ such that $\Omega_{R}^{n+p}(M) \oplus P \cong \Omega_{R}^{n}(M) \oplus Q$ in $R\Mod$.
Taking their syzygies, we obtain an isomorphism $\Omega_{R}^{n+p+1}(M) \cong \Omega_{R}^{n+1}(M)$ in $R\Mod$.
\end{proof}

We are now ready to give the main result of this section, which says that periodic dimension is almost equal to Gorenstein projective dimension when both of the two dimensions are finite.

\begin{theorem} \label{claim_6}
Let $M$ be an eventually periodic $R$-module of finite Gorenstein projective dimension $r$.
Then we have
\begin{align*}
    r \,\leq\, \Pdim_R M \,\leq\, r+1.
\end{align*}
Moreover, there exists an isomorphism in $R\sMod$
\begin{align*}
     \Omega_{R}^{r+p}(M) \cong \Omega_{R}^{r}(M),
\end{align*}
where $p$ is the period of the first periodic syzygy of $M$.
\end{theorem} 

\begin{proof}
Suppose that $M$ is $(n, p)$-eventually periodic.
We first show that $r \leq n \leq r+1$.
Fix a minimal projective resolution $P_\bullet \rightarrow M$ of $M$. 
The inequality $r \leq n$ is obtained from the fact that $\Omega_{R}^{n}(M) \cong \Omega_{R}^{n+ip}(M)$ for all $i \geq 0$.
On the other hand, splicing the periodic part 
\[ 0 \rightarrow \Omega_{R}^{n+p}(M) \rightarrow P_{n+p-1} \rightarrow \cdots \rightarrow P_{n} \rightarrow \Omega_{R}^{n}(M) \rightarrow 0\] 
repeatedly, we can construct an acyclic complex $T_\bullet$ of projective $R$-modules such that  $\Omega_0(T_\bullet) = \Omega_{R}^{n}(M)$.
Since $\Omega_{R}^{i}(M)$ is Gorenstein projective for any $i \geq n$, \cite[Lemma 2.3.3]{Veliche_2006} implies that the acyclic complex $T_\bullet$ becomes totally acyclic.
Hence $\Sigma^{i}(\Omega_{R}^{n}(M))= \Sigma^{i}(\Omega_0(T_\bullet))= \Omega_{-i}(T_\bullet)$ for all  $i \in \mathbb{Z}$, where $\Sigma$ denotes the shift functor on $R\sGProj$.
Since $\Sigma^{-1} = \Omega_R$, there exist isomorphisms in $R\sGProj$
\begin{align*} 
    \Omega_{R}^{r}(M) 
    \cong \Sigma^{n-r} \Sigma^{r-n}(\Omega_{R}^{r}(M)) 
    \cong \Sigma^{n-r}(\Omega_{R}^{n}(M)) 
    \cong \Omega_{R}^{n+l}(M)
\end{align*}
for some $l$ with $0 \leq l < p$.
Applying $\Omega_R^{p}$ to the above, we obtain the following isomorphisms in $R\sGProj$:
\begin{align*}
    \Omega_{R}^{r+p}(M) 
    \cong \Omega_{R}^{n+l+p}(M) 
    \cong \Omega_{R}^{n+l}(M) 
    \cong \Omega_{R}^{r}(M).
\end{align*}
Thus Lemma \ref{claim_1} shows that $n \leq r+1$.
This completes the proof.
\end{proof}

As in Theorem \ref{claim_6}, one can prove a similar result for a Noetherian semiperfect ring. 
We state it without proof.

\begin{theorem} \label{claim_50}
Let $R$ be a Noetherian semiperfect ring and $M$ a finitely generated $(n,p)$-eventually periodic $R$-module with $\Gdim_R M = r < \infty$.
Then we have
\begin{align*}
    r \,\leq\, n \,\leq\, r+1.
\end{align*}
Moreover, there exists an isomorphism in $R\smod$
\begin{align*}
     \Omega_{R}^{r+p}(M) \cong \Omega_{R}^{r}(M).
\end{align*}
\end{theorem} 

\begin{remark} \label{remark_3}
{\rm
Let $R$ be a Gorenstein local ring.
Then Theorem \ref{claim_50} can be used to improve results on eventually periodic $R$-modules such as \cite[Theorem 4.1]{Eisenbud_1980}, \cite[Theorem 1.2]{Gasharov-Peeva_1990} and \cite[Theorem 1.6]{Avramov_1989_proceedings}.
}\end{remark}

We end this section with three corollaries of Theorem \ref{claim_6}. 
First, we consider two extreme cases for eventually periodic modules having finite Gorenstein projective dimension.

\begin{corollary} \label{claim_7}
The following statements hold for any $R$-module $M$.
\begin{enumerate}

    \item If $M$ is  $p$-periodic and has finite Gorenstein projective dimension, then $M$ is $p$-strongly Gorenstein projective.
    
    \item If $M$ is $(n,p)$-eventually periodic and is Gorenstein projective, then $M$ is $p$-strongly Gorenstein projective. 

\end{enumerate}
\end{corollary}
\begin{proof} 
It is a direct consequence of Theorem \ref{claim_6}.
\end{proof}

Next, we refine Theorem \ref{claim_6} in case $R$ is left artin.

\begin{corollary} \label{claim_49}
Let $R$ be a left artin ring and $M$ a finitely generated $(n,p)$-eventually periodic $R$-module with $\Gpd_R M = r < \infty$.
Then the following assertions hold.
\begin{enumerate}
    \item $r \leq n \leq r+1$, and $\Omega_{R}^{r+p}(M) \oplus P \cong \Omega_{R}^{r}(M)$ in $R\mod$ for some $P \in R\proj$.
    \item $n=r$ if and only if  $\Omega_{R}^{r}(M)$ is in $R\modp$.
    \item If we write $\Omega_{R}^{n-1}(M) = X \oplus Q$ with $X \in R\modp$ and $Q \in R\proj$, then  $r=n-1$ if and only if  $X \cong \Omega_{R}^{n+p-1}(M)$ in $R\mod$.
\end{enumerate}
\end{corollary}
\begin{proof} 
For (1), it suffices to get the stated isomorphism. 
We know from Theorem \ref{claim_6} that $\Omega_{R}^{r+p}(M)  \cong \Omega_{R}^{r}(M)$ in $R\sGproj$.
Then the fact that $\Ext_{R}^{n}(M, R)= 0$ for all $i>r$ implies that $\Omega_R^{r+p}(M)$ has no non-zero projective direct summand.
Consequently, the Krull-Schmidt theorem shows that there exists a finitely generated projective $R$-module $P$ such that $\Omega_{R}^{r+p}(M) \oplus P \cong \Omega_{R}^{r}(M)$ in $R\mod$.
Now, (2) is an immediate consequence of (1). 

It remains to show (3). 
Write the inequalities in (1) as $n-1 \leq r \leq n$.  
Let $\Omega_{R}^{n-1}(M) = X \oplus Q$ be a decomposition as $R$-modules, where $X$ and $Q$ belong to $R\modp$ and $R\proj$, respectively. 
Such an $R$-module $X$ is uniquely determined up to isomorphism.
To verify the\,\lq\lq only if\,\rq\rq\,part, suppose that $r= n-1$.
We then see that $X$ is a Gorenstein projective module satisfying that $\Omega_{R}(X) \cong \Omega_{R}^{n}(M) \cong \Omega_{R}^{n+p}(M) \cong \Omega_{R}^{p+1}(X)$ in $R\smod$. 
Thus it follows from Lemma \ref{claim_1} and Corollary \ref{claim_7} that $X$ is $p$-strongly Gorenstein projective, which implies that  $X \cong \Omega_{R}^{p}(X)$ in $R\smod$. 
This isomorphism can be lifted to $R\mod$ because both $X$ and $\Omega_{R}^{p}(X)$ have no non-zero projective summand.
Therefore, we obtain $X \cong  \Omega_{R}^{p}(X) \cong  \Omega_{R}^{n+p-1}(M)$ in $R\mod$.
The\,\lq\lq if\,\rq\rq\,part follows from the inequalities $n-1 \leq r \leq n \leq n+p-1$.
\end{proof}

The next example explains how to use Corollary \ref{claim_49} (3).
Recall that an artin algebra is called {\it CM-finite} \cite{Beligiannis_2011} if there are only finitely many pairwise nonisomorphic finitely generated indecomposable  Gorenstein projective modules. 
The class of CM-finite algebras includes algebras of finite global dimension, representation-finite algebras, and monomial algebras. 
Observe that if $\L$ is a CM-finite Gorenstein algebra, then any finitely generated $\L$-module has finite periodic and Gorenstein projective dimension (use Proposition \ref{Veliche_2006_2.4.2} and \cite[Lemma 2.2]{X-WChen_2012} for example).

\begin{example} \label{example_5}
{\rm 
Let $\L = kQ/I$, where $Q$ is the following quiver:
\begin{align*}
    \xymatrix@C=20pt@R=10pt{
      & 4 \ar[rd]^-{\beta}  &  &  &   &  \\
    5 \ar@(ul,dl)_-{\delta} \ar[ru]^-{\gamma} &  & 3 \ar[ll]^-{\alpha} \ar[r]^-{\varepsilon} & 2 \ar[r]^-{\zeta} & 1 \ar[r]^-{\eta} & 0 
    }
\end{align*}
and $I$ is the ideal of $kQ$ generated by $\alpha\beta\gamma -\delta^2, \, \gamma\alpha\beta, \, \delta\alpha, \, \gamma\delta, \, \varepsilon\beta, \, \zeta\varepsilon, \, \eta\zeta$.
Then one sees that $\L$ is a finite dimensional special biserial $4$-Gorenstein algebra. 
Moreover, it follows from \cite[Section 2]{Wald-Waschbusch_1985} that $\L$ is representation-finite.
Thus the above observation concludes that $\Pdim_{\L}M < \infty$ and  $\Gpd_{\L}M < \infty$ for all $M \in \L\mod$.
Recall that $S_i$ is the simple $\L$-module corresponding to the vertex $i$.
A calculation shows that
\begin{align*}
    \projdim_{\L}S_i = 
    \begin{cases}
    i
    & 
    \mbox{ if $ 0 \leq i \leq 2$, }\\
    \infty & \mbox{ if $3 \leq i \leq 5$}
    \end{cases}
\end{align*}
and 
\begin{align*}
    \Pdim_{\L}S_i = 
    \begin{cases}
    i+1 & 
    \mbox{ if $ 0 \leq i \leq 4$, }\\
    1 & \mbox{ if $i = 5$}.
    \end{cases}
\end{align*}
Moreover, the period of the first periodic syzygy of $S_i$ is equal to $1$ if $i = 0, 1, 2, 4, 5$ and to $2$ if $i=3$.
We denote by $P_i$ a projective cover of $S_i$ and by $\rad M$ the radical of a $\L$-module $M$.
By direct computation, we have $\Omega_{\L}^{i}(S_i) = 0 \oplus P_0$ and $\Omega_{\L}^{i+1}(S_i) = 0$ for $i=0, 1, 2$; $\Omega_{\L}^{3}(S_3) = (\rad P_3)/S_4 \oplus P_0$ with $(\rad P_3)/S_4 \in \L\modp$ and $\Omega_{\L}^{5}(S_3) = (\rad P_3)/S_4 $; and $\Omega_{\L}^{4}(S_4) = \rad P_5 \oplus P_0$ with $\rad P_5 \in \L\modp$ and $\Omega_{\L}^{5}(S_4) = \rad P_5$.    
On the other hand, $\Omega_{\L}^{0}(S_5) = S_5 \not= \rad S_5 = \Omega_{\L}^{1}(S_5)$ with $S_5 \in \L\modp$.
Therefore, Corollary \ref{claim_49} (3) concludes that 
\begin{align*}    
    \Gpd_{\L}S_i = 
    \begin{cases}
    \Pdim_{\L}S_i -1  = i & 
    \mbox{ if $ 0 \leq i \leq 4$, }\\
    \Pdim_{\L}S_5 = 1 & \mbox{ if $i = 5$}.
    \end{cases}
\end{align*}
}\end{example}

Finally, we give a useful property of periodic dimensions.

\begin{corollary} \label{claim_48}
Let $\{ M_i \}_{i \in I}$ be a finite set of finitely generated modules over a left artin ring $R$.
If $M := \bigoplus_{i\in I} M_i$ has finite Gorenstein projective dimension, then we have
    \begin{align*}
        \Pdim_R M = \sup\left\{\, \Pdim_R M_i \mid i \in I\,\right\}.
    \end{align*}
\end{corollary}
\begin{proof} 
Together with the inequality (\ref{eq_7}), Proposition \ref{claim_46} (1) implies that 
$\Pdim_R M$ $= \infty$ if and only if $\Pdim_R M_i$ $= \infty$ for some $i \in I$.
Thus we have to obtain the desired equality in case $\Pdim_R M = n < \infty$.
In this case, the first periodic syzygy $\Omega_R^n(M)$ is Gorenstein projective by Theorem \ref{claim_6} and hence belongs to $R\modfpd$. Here, we use the fact that $R\GProj \subseteq {}^{\perp}(R\Proj) = {}^{\perp}(R\Fpd)$.
Then Corollary \ref{claim_47} completes the proof.
\end{proof}

\begin{example} \label{example_6}
{\rm 
Let $\L = kQ/I$ be as in Example \ref{example_5}.
We know that $\Pdim_{\L}M < \infty$ and  $\Gpd_{\L}M < \infty$ for all $M \in \L\mod$.
Recall that $J(\L)$ is the Jacobson radical of $\L$.
Together with the fact that $\L/J(\L) = \bigoplus_{i=1}^{5} S_i$, we can conclude from Corollary \ref{claim_48} and Example \ref{example_5} that 
\begin{align*}    
    \Pdim_{\L}\L/J(\L) = \sup\left\{\, \Pdim_{\L}S_i \mid 0 \leq i \leq  5\,\right\}= 5.
\end{align*}
}\end{example}


\section{Bimodule periodic dimensions of algebras} \label{The_case_of_algebras} 

In the rest of this paper, an algebra will mean a finite dimensional algebra over the field $k$.
This section studies the periodic dimensions of algebras viewed as regular bimodules.  
We start with the definition of eventually periodic algebras.

\begin{definition} \label{def_4}
{\rm 
An algebra $\L$ is called {\it eventually periodic} if the regular $\L$-bimodule $\L$ is eventually periodic.
If $\L$ is periodic as a $\L$-bimodule, $\L$ is said to be {\it periodic}.
}
\end{definition}

Throughout this paper, an {\it $(n,p)$-eventually periodic algebra} will mean an algebra $\L$ that is $(n,p)$-eventually periodic over $\Le$.
Further, we refer to $n = \Pdim_{\Le}\L$ as the {\it bimodule periodic dimension} of $\L$.

We now make a brief note on eventually periodic algebras:
as pointed out in \cite[Section 2]{Usui_2022}, eventually periodic algebras are not Gorenstein in general.
This is perhaps slightly surprising because periodic algebras are known to be self-injective (\cite[Proposition IV.11.18]{Skowronski-Yamagara_book_I}).
For this reason, we will characterize eventually periodic Gorenstein algebras (Proposition \ref{claim_39} and Theorem \ref{claim_38}).
We will also show that eventually periodic algebras are at least both left and right weakly Gorenstein (Proposition \ref{claim_40}). 
On the other hand, the class of eventually periodic algebras includes monomial Gorenstein algebras (\cite[the proof of Corollary 6.4]{Dotsenko-Gelinas-Tamaroff_2023}) and monomial Nakayama algebras (\cite[Section 3.2]{Usui_2022}). 
It is not difficult to check that this is a consequence of the following result due to K{\" u}pper \cite[Corollary 2.10 (1) and (2)]{kupper2010two}.

\begin{proposition}[K{\" u}pper] \label{claim_51}
Let $\L$ be a monomial algebra. 
Then $\L$ is an eventually periodic algebra if and only if every simple $\L$-module is eventually periodic.
\end{proposition}

We now move on to considerations on the bimodule periodic dimensions of eventually periodic algebras.
In \cite[the proof of Corollary 6.4]{Dotsenko-Gelinas-Tamaroff_2023}, Dotsenko, G\'{e}linas and Tamaroff showed that $\Pdim_{\Le} \L \leq d+1$ for any monomial $d$-Gorenstein algebra $\L$, and 
the author proved in \cite[Proposition 4.2]{Usui21_No02} that $\Pdim_{\Le} \L \geq d$ for any $d$-Gorenstein algebra $\L$.
As the main result of this section, we will show that the bimodule periodic dimension of an eventually periodic $d$-Gorenstein algebra equals either $d$ or $d+1$. 
To this end, we now calculate the bimodule Gorenstein dimension of an arbitrary Gorenstein algebra.

Let $\L$ be an algebra. 
We see from \cite[Lemma 8.2.4]{Witherspoon_book} that for any finitely generated $\L$-modules $M$ and $N$, there exists an isomorphism of graded vector spaces 
\begin{align*}
    \Ext_\L^{\bullet}(M, N)  \cong \Ext_{\Le}^{\bullet}(\L, \Hom_k(M, N)).
\end{align*}
Let $D$ denote the $k$-duality $\Hom_k(-, k)$. 
Then the isomorphism of $\Le$-modules 
\begin{align*}
    \Le = {}_{\L}\L \otimes \L_\L \cong \Hom_k(D(\L_\L), {}_{\L}\L)
\end{align*}
induces the following isomorphism of graded vector spaces:
\begin{align} \label{eq_6}
    \Ext_\L^{\bullet}(D(\L), \L) \cong \Ext_{\Le}^{\bullet}(\L, \Le).
\end{align}

\begin{proposition} \label{claim_44}
Let $\L$ be a Gorenstein algebra. Then we have
\[\Gdim_{\Le} \L= \injdim_{\L}\L.\]
\end{proposition}

\begin{proof}
Assume that $\L$ is $d$-Gorenstein.
Since the enveloping algebra $\Le$ is $(2d)$-Gorenstein, it follows that $\Gdim_{\Le} \L \leq 2d$.
Moreover, we obtain that 
\begin{align*}
   d  = \injdim_{\L^{\rm op}}\L = \projdim_\L D(\L) =  \Gdim_\L D(\L).
\end{align*}
Hence the isomorphism (\ref{eq_6}) implies that $\Gdim_{\Le} \L = \Gdim_\L D(\L) = d$.
\end{proof}

Remark that the proposition extends a result of Shen \cite[Proposition 5.6]{Shen_2019} to the higher dimensional case.

We are now ready to prove the following main result of this section.

\begin{theorem} \label{claim_10}
Let $\L$ be an $(n, p)$-eventually periodic $d$-Gorenstein algebra.
Then we have
\begin{align*}
    d \leq  n \leq d+1.
\end{align*}
Moreover, there exists an isomorphism in $\Le\mod$
\begin{align*}
    \Omega_{\Le}^{d+p}(\L) \oplus P \cong \Omega_{\Le}^d(\L) 
\end{align*}
for some $P$ in $\Le\proj$.
In particular, $n = d$ if and only if $\Omega_{\Le}^d(\L)$ has no non-zero projective direct summand.
\end{theorem} 
\begin{proof}
It follows from the formula (\ref{eq_11}) and Proposition \ref{claim_44} that the finitely generated eventually periodic $\Le$-module $\L$ satisfies $\Gpd_{\Le} \L = \Gdim_{\Le} \L= d$. 
Thus Corollary \ref{claim_49} completes the proof. 
\end{proof}

\begin{remark} \label{remark_1}
{\rm
It is possible to describe the bimodule periodic dimension of an eventually periodic algebra with finite bimodule Gorenstein dimension.
As will be seen in Theorem \ref{claim_38}, such algebras are Gorenstein.
}\end{remark}

\begin{remark} \label{remark_2}
{\rm
    The bound given in Theorem \ref{claim_10} is the best possible. 
    Indeed, $d$-Gorenstein algebras of bimodule periodic dimension $d$ can be found in \cite[Proposition 4.3]{Usui21_No02}.
    Besides, Proposition \ref{claim_52} and Examples \ref{example_4} and \ref{example_1} below exhibit examples of $d$-Gorenstein algebras of bimodule periodic dimension $d+1$.
}\end{remark}

We now briefly recall some basic facts on projective resolutions over an algebra $\L$.
Let $P_\bullet \xrightarrow{\varepsilon} \L$ be  a projective resolution of $\L$ over $\Le$.
Then any $\L$-module $M$ admits a  projective resolution of the form  
\begin{align*}
    P_{\bullet} \otimes_{\L} M \xrightarrow{\varepsilon \otimes_{\L} \id_M} \L \otimes_{\L} M.
\end{align*}
In particular, $\Omega_{i}(P_{\bullet} \otimes_{\L} M) = \Omega_{i}(P_{\bullet}) \otimes_{\L} M$ for all $i \geq 0$.
Similar projective resolutions can be constructed for any $\L^{\rm op}$-modules.  
Therefore, one gets an inequality
\begin{align*}
    \gldim \L \leq \projdim_{\Le}\L.
\end{align*}
It is known that the equality holds if the semisimple quotient $\L/J(\L)$ is separable. 
The following observation gives another condition under which the equality holds.

\begin{proposition} \label{claim_52}
Let $\L$ be an algebra with finite bimodule projective dimension $d$. 
Then $\L$ is a $(d+1, 1)$-eventually periodic $d$-Gorenstein algebra. 
Moreover, we have $\gldim \L = \projdim_{\Le}\L$.
\end{proposition}

\begin{proof}
We claim that $\L$ is $d$-Gorenstein.
Since $\gldim \L \leq \projdim_{\Le}\L < \infty$, it follows that $\L$ is Gorenstein with $\injdim_\L \L = \gldim\L$.
Now, one has
\begin{align*}
d  = \projdim_{\Le}\L = \Gdim_{\Le} \L = \injdim_\L \L ,
\end{align*}
where the third equality follows from Proposition \ref{claim_44}. This completes the proof.
\end{proof}

Next, we have the following description of the bimodule periodic dimensions of eventually periodic Gorenstein algebras.
In what follows, we set $\mathbbm{k} := \L/J(\L)$ for an algebra $\L$.

\begin{theorem} \label{claim_45}
Let $\L$ be an eventually periodic $d$-Gorenstein algebra.
If $\mathbbm{k}$ is a separable algebra, then we have
\begin{align*}
    \Pdim_{\Le} \L = \Pdim_{\L}\mathbbm{k} = \Pdim_{\L^{\rm op}}\mathbbm{k},
\end{align*}
where the common value is either $d$ or $d+1$.
Moreover, the following conditions are equivalent.
\begin{enumerate}
  \setlength{\itemsep}{.5mm} 
    \item The bimodule periodic dimension of $\L$ is equal to $d$.
    \item The $d$-th syzygy of  ${}_{\Le}\L$ is in $\Le\modp$.
    \item The $d$-th syzygy of  ${}_{\L}\mathbbm{k}$ is in $\L\modp$.
    \item The $d$-th syzygy of  ${}_{\L^{\rm op}}\mathbbm{k}$ is in $\L^{\rm op}\modp$.
\end{enumerate}
\end{theorem} 

\begin{proof}
We need only verify that $\Pdim_{\Le} \L = \Pdim_{\L}\mathbbm{k} = \Pdim_{\L^{\rm op}}\mathbbm{k}$.
It follows from Propositions \ref{Veliche_2006_2.4.2} and \ref{claim_44} that $\Gdim_{\Le}\L = \Gdim_{\L}\mathbbm{k} = \Gdim_{\L^{\rm op}}\mathbbm{k} = d$. 
Moreover, $\Pdim_{\Le} \L$ is finite by definition, and $\Pdim_{\L}\mathbbm{k}$ and $\Pdim_{\L^{\rm op}}\mathbbm{k}$ are both finite by Lemma \ref{claim_42}, which will be proved in the next section.
Thus Corollary \ref{claim_49} implies that 
\begin{align*}
    d \  \leq  \ \Pdim_{\Le} \L, \ \Pdim_{\L}\mathbbm{k}, \ \Pdim_{\L^{\rm op}}\mathbbm{k} \ \leq \  d+1.
\end{align*}

We claim that $\Pdim_{\Le} \L = d+1$ implies $\Pdim_\L \mathbbm{k}= d+1 =\Pdim_{\L^{\rm op}} \mathbbm{k}$.
Since $\mathbbm{k}$ is separable, one has that
\begin{align*}
    J(\Le) = J(\L)\otimes_k \L^{\rm op} + \L \otimes_k J(\L^{\rm op}) \mbox{\quad and \quad } \Le /J(\Le) \cong \mathbbm{k}\otimes_k \mathbbm{k}^{\rm op}.
\end{align*}
Hence if $P_\bullet \rightarrow \L$ is a minimal projective resolution of $\L$ over $\Le$, then the following complex induced by the tensor functor $\Le /J(\Le) \otimes_{\Le} -$ has trivial differentials:
\begin{align*} 
    \mbox{$\mathbbm{k} \otimes_{\L}  P_\bullet \otimes_{\L} \mathbbm{k} \rightarrow \mathbbm{k} \otimes_{\L} \L \otimes_{\L} \mathbbm{k}$.}
\end{align*}
This implies that the projective resolutions 
\begin{align}\label{eq_10}
    \mbox{
$P_{\bullet} \otimes_{\L} \mathbbm{k} \rightarrow \L \otimes_{\L} \mathbbm{k} = {}_{\L}\mathbbm{k}$ \quad  and \quad  $\mathbbm{k} \otimes_{\L} P_{\bullet}  \rightarrow \mathbbm{k} \otimes_{\L} \L = {}_{\L^{\rm op}}\mathbbm{k}$.}
\end{align}
are both minimal. 
We thus conclude that if $\Omega_{\Le}^d(\L)$ has a non-zero projective direct summand, then so do 
\begin{align*} 
    \Omega_{\L}^d(\mathbbm{k}) = \Omega_{\Le}^{d}(\L) \otimes_{\L} \mathbbm{k}  \quad  \mbox{ and } \quad   \Omega_{\L^{\rm op}}^d(\mathbbm{k}) = \mathbbm{k} \otimes_{\L} \Omega_{\Le}^{d}(\L).
\end{align*}
Then Corollary \ref{claim_49} enables us to obtain the desired statement.

To complete the proof, it is enough to check that $\Pdim_{\Le} \L = d$ implies  $\Pdim_{\L}\mathbbm{k} = d = \Pdim_{\L^{\rm op}} \mathbbm{k}$. 
But, this is trivial because of the minimality of the projective resolutions (\ref{eq_10}).
\end{proof}

Note that $\mathbbm{k}$ is separable when $\L$ is an algebra over a perfect field or a bound quiver algebra over a field.
We end this section with an example of $d$-Gorenstein algebras $\L$ with $\Pdim_{\Le} \L = d+1$ and $\projdim_{\Le} \L = \infty$.

\begin{example} \label{example_4}
{\rm 
Consider the following disconnected quiver $Q$:
\begin{align*}
    \xymatrix{0 \ar@(ul,dl)_-{\beta}  & -1
    }
\end{align*}
Let $I$ be the ideal of $kQ$ generated by $\beta^2$, and let $\L = kQ/I$.
Since $\L$ is a monomial algebra isomorphic to the product of the periodic algebra $k[x]/(x^2)$ and the simple self-injective algebra $k$, it follows that  $\L$ is self-injective and hence eventually periodic.
Recall that we denote by $S_i$ the simple $\L$-module corresponding to the vertex $i$.
Since 
\begin{align*}
    \projdim_{\L}S_i = 
    \begin{cases}
    0 & 
    \mbox{ if $i = -1$, }\\
    \infty & \mbox{ if $i = 0$}
    \end{cases}
    \quad \mbox{and} \quad 
    \Pdim_{\L}S_i = 
    \begin{cases}
    1 & 
    \mbox{ if $i = -1$, }\\
    0 & \mbox{ if $i = 0$},
    \end{cases}
\end{align*}
we have that
\begin{align*}
    \Pdim_{\Le} \L = \Pdim_{\L}\mathbbm{k} = 
    \max\left\{\, \Pdim_{\L}S_i \,\mid\, i = -1, 0 \,\right\}
    = 1,
\end{align*}
where the first and the second equality are obtained from Theorem \ref{claim_45} and Corollary \ref{claim_48}, respectively. 
}\end{example}

The next example is inspired by \cite[Section 2.3]{Dotsenko-Gelinas-Tamaroff_2023}.

\begin{example} \label{example_1}
{\rm 
For any positive integer $d$, consider the following quiver $Q$:
\begin{align*}
\xymatrix{
    d \ar@(ul,dl)_-{\beta}  \ar[r]^-{\alpha_{d}} & d-1 \ar[r]^-{\alpha_{d-1}}& d-2 \ar[r]  & \cdots \ar[r] & 1\ar[r]^-{\alpha_{1}} & 0
    }
\end{align*}
Let $I$ be the ideal of $kQ$ generated by $\{ \beta^2, \alpha_{i-1}\alpha_{i} \mid 2 \leq i \leq d  \}$, and let $\G = kQ/I$. 
Thanks to \cite[Theorem 2.9]{Dotsenko-Gelinas-Tamaroff_2023}, it follows that the monomial algebra $\G$ is $d$-Gorenstein and hence eventually periodic.
Moreover, one has that 
\begin{align*}
    \projdim_{\G}S_i = 
    \begin{cases}
    i & 
\mbox{ if $0 \leq i \leq  d-1$, }\\
    \infty & \mbox{ if $i = d$}
    \end{cases}
\end{align*}

and 

\begin{align*}
    \Pdim_{\G}S_i = 
    \begin{cases}
    i+1 & 
    \mbox{ if $0 \leq i \leq  d-1$, }\\
    d+1 & \mbox{ if $i = d$}.
    \end{cases}
\end{align*}
Note that the simple projective $\G$-module $S_0 ( = \G \alpha_1)$ is a non-zero projective direct summand of $\Omega_{\G}^{d}(S_d)$.
As in Example \ref{example_4}, one concludes that
\begin{align*}
    \Pdim_{\Ge}\G = \Pdim_{\G}\mathbbm{k} = \max\left\{\, \Pdim_{\G}S_i \mid 0 \leq i \leq  d\,\right\}= d+1.
\end{align*}
}\end{example}


\section{Homological properties of eventually periodic algebras} \label{Homological properties of eventually periodic algebras}

This section reveals some basic homological properties of eventually periodic algebras. 
We show that a lot of homological conjectures hold for this class of algebras. 
Moreover, as promised in Section \ref{The_case_of_algebras}, we characterize eventually periodic Gorenstein algebras and show that eventually periodic algebras are both left and right weakly Gorenstein.

First of all, we focus on the {\it periodicity conjecture}, which states that an algebra must be periodic if all its simple modules are periodic.
We refer to \cite[Section 1]{Erdmann-Skowronski_2015} for more information on the conjecture.

\begin{proposition} \label{claim_3}
An eventually periodic connected algebra $\L$ is periodic if and only if all the simple $\L$-modules are periodic. 
\end{proposition}
\begin{proof}
It suffices to show the\,\lq\lq if\ \!\rq\rq\,part.
By \cite[Theorem 1.4]{Green-Snashall-Solberg_2003}, the algebra $\L$ satisfying the required condition is self-injective, which implies that $\Gdim_{\Le}\L = 0$. 
Applying Corollary \ref{claim_49} (2) to the indecomposable Gorenstein projective $\Le$-module $\L$, we conclude that $\L$ is periodic.
\end{proof}

We now prepare the following easy lemma, which will be frequently used from now on.

\begin{lemma} \label{claim_42}  
Let $\L$ be an $(n,p)$-eventually periodic algebra. Then we have the following statements.
\begin{enumerate}
    \item The endofunctor $\Omega_{\L}$ on $\L\sMod$ satisfies that $\Omega_{\L}^{n+p} \cong \Omega_{\L}^{n}$.
    \item The endofunctor $\Omega_{\Lop}$ on $\Lop\sMod$ satisfies that $\Omega_{\Lop}^{n+p} \cong \Omega_{\Lop}^{n}$.
\end{enumerate} 
In particular, for any $\L$-module $M$ (resp. $\Lop$-module $N$), we have  
\begin{align*}
    \Pdim_{\L} M \leq n+1 \quad \left(\mbox{resp. } \Pdim_{\L^{\rm op}} N \leq n+1 \right).
\end{align*}
Moreover, the period of the first periodic syzygy of $M$ (resp. $N$) divides $p$.
\end{lemma}
\begin{proof} 
We only prove (1); the proof of (2) is similar. 
Since $\Omega_\L^i \cong \Omega_{\Le}^{i}(\L) \otimes_\L -$ as endofunctors on $\L\sMod$ for every $i \geq 0$, 
there are isomorphisms of endofunctors on $\L\sMod$ \[\Omega_\L^{n+p} \cong \Omega_{\Le}^{n+p}(\L) \otimes_\L - \cong \Omega_{\Le}^{n}(\L) \otimes_\L -  \cong \Omega_\L^{n}.\]
The remaining is a consequence of Lemma \ref{claim_1}.
\end{proof}

The lemma enables us to decide whether an eventually periodic algebra is Gorenstein or not.

\begin{proposition} \label{claim_39} 
Let $\L$ be an eventually periodic algebra. 
Then the following conditions are equivalent.
\begin{enumerate}
    \item  $\L$ is a Gorenstein algebra.
    \item A finitely generated $\L$-module $M$ is periodic if and only if $M$ is Gorenstein projective without non-zero projective direct summands.
\end{enumerate}
\end{proposition}
\begin{proof}
We first prove that (1) implies (2).
It follows from Proposition \ref{Veliche_2006_2.4.2} and Lemmas \ref{claim_1} and \ref{claim_42} that any finitely generated $\L$-modules $M$ satisfy that $\Gdim_\L M < \infty$ and $\Pdim_\L M < \infty$.
Therefore, the desired equivalence is a consequence of Corollary \ref{claim_49}.
Conversely, suppose that the equivalence in (2) holds.  
Since we know by Lemma \ref{claim_42} that $\Pdim_\L \L/J(\L) < \infty$, the equivalence implies that $\Gdim_\L \L/J(\L)  < \infty$. 
Thus Proposition \ref{Dotsenko-Gelinas-Tamaroff_2023_Proposition 2.4} finishes the proof.
\end{proof}

Recall that the {\it big finitistic dimension} of an algebra  $\L$ is defined as 
\[\Findim \L := \sup\{\projdim_{\L}M \mid  M \in \L\Mod \ \mbox{ and }\ \projdim_{\L}M <\infty \}\]
and the {\it little finitistic dimension} of $\L$ is defined to be 
\[\findim \L := \sup\{\projdim_{\L}M \mid  M \in \L\mod \ \mbox{ and }\ \projdim_{\L}M <\infty \}.\] 
It is conjectured that the little finitistic dimension of an arbitrary algebra is finite. 
This is known as the {\it finitistic dimension conjecture} and is still open. 
See \cite{Smalo_2000,Yamagata_book_1996,Zimmermann-Huisgen_1992} for more information on this and related homological conjectures.
We now observe that the finitistic dimension conjecture holds for eventually periodic algebras and their opposite algebras.

\begin{proposition}\label{claim_18}  
Let $\L$ be an $(n,p)$-eventually periodic algebra.  
Then 
\begin{align*}
    \Findim \L \leq n \quad \mbox{ and } \quad  \Findim \Lop \leq n.
\end{align*}
\end{proposition} 

\begin{proof} 
We only show that $\Findim \L \leq n$; the other is similarly proved.
Let $M$ be a $\L$-module of finite projective dimension.
Lemma \ref{claim_42} implies that $\Omega_{\L}^{n}(M) \cong \Omega_{\L}^{n+ip}(M)$ in $\L\sMod$ for all $i \geq 0$, so that $\Omega_{\L}^{n}(M)$ must be projective.
\end{proof}

The following consequence of Proposition \ref{claim_18} says that {\it Gorenstein symmetric conjecture} \cite{Beligiannis_2005} holds for eventually periodic algebras.

\begin{proposition} \label{claim_19}  
Let $\L$ be an eventually periodic algebra. 
Then $\injdim_{\L}\L < \infty$ if and only if $\injdim_{\L^{\rm op}}\L < \infty$.
\end{proposition}

\begin{proof}
It is a  consequence of Proposition \ref{claim_18} and \cite[Proposition 6.10]{AusRei_1991_AdvMath}.
\end{proof}

We say that an algebra $\L$ {\it satisfies ${(\rm AC)}$} if  the following condition is satisfied:
\begin{enumerate}
    \item[(AC)] For a finitely generated $\L$-module $M$, there exists an integer $b_{M} \geq 0$ such that if a finitely generated $\L$-module $N$ satisfies that $\Ext_{\L}^{\gg 0}(M, N) = 0$, then $\Ext_{\L}^{> b_M}(M, N) = 0$.
\end{enumerate}  
See \cite{Celikbas-Takahashi_2013,Christensen-Holm_2010} for more information on this condition and related homological problems.

\begin{proposition} \label{claim_28}   
An eventually periodic algebra and its opposite algebra satisfy {\rm (AC)}.
\end{proposition}

\begin{proof} 
We only prove that an $(n, p)$-eventually periodic algebra $\L$ satisfies {\rm (AC)}; the proof for the opposite algebra $\Lop$ is similar.
Since  $\findim \L < \infty$ by Proposition \ref{claim_18}, it suffices to consider finitely generated $\L$-modules $M$ with $\projdim_{\L}M = \infty$.
It follows from Lemma \ref{claim_42} that $\Ext_{\L}^{i}(M, N) \cong \Ext_{\L}^{i+p}(M, N)$ for all $i > n$ and all $N \in \L\mod$. 
Thus taking $b_M := n$ will complete the proof.
\end{proof}

Thanks to Christensen and Holm \cite[Theorem A]{Christensen-Holm_2010}, we know that the following condition holds for an algebra $\L$ satisfying (AC):
\begin{enumerate}
    \item[(ARC)] Let $M$ be a finitely generated $\L$-module. If $\Ext_{\L}^{i}(M, M) $ $=$ $0$ $= \Ext_{\L}^{i}(M,  \L)$ for all $i >0$, then  $M$ is projective.
\end{enumerate}  
On the other hand, we know from Subsection \ref{Preliminaries_GorensteinRings} that if $\L$ is a Gorenstein algebra, then the bimodule Gorenstein dimension of $\L$ is finite. 
The following main result of this section shows that the converse holds for the class of eventually periodic algebras.
The key to the proof is the condition (ARC).

\begin{theorem} \label{claim_38} 
Let $\L$ be an eventually periodic algebra.
Then $\L$ is a Gorenstein algebra if and only if the Gorenstein dimension of the regular $\L$-bimodule $\L$ is finite.
In this case, we have $\Gdim_{\Le} \L= \injdim_{\L}\L$.
\end{theorem}

\begin{proof}  It is sufficient to show the\ \lq\lq if\ \!\rq\rq\ part.
Suppose that $\L$ is $(n, p)$-eventually periodic with $\Gdim_{\Le}\L = r < \infty$. 
Then  the isomorphism (\ref{eq_6}) implies that $\Ext_\L^{i}(D(\L), \L)  = 0$ for all $i >r$, and we know from Corollary \ref{claim_49} that $r\leq n$.
Moreover, by Lemma \ref{claim_42}, there exists an isomorphism 
\begin{align*}
    \Ext_{\L}^{i}\!\left(\Omega_{\L}^{n}(D(\L)), N\right) \cong \Ext_{\L}^{i+p}\!\left(\Omega_{\L}^{n}(D(\L)), N\right)
\end{align*}
for all $i >  0$ and all $N \in \L\mod$. 
As a result, letting  $m$ be an integer divided by $p$ with $m > n$, we have the following isomorphisms for all $i>0$:
\begin{align*}
    \Ext_\L^{i}\!\left( \Omega_{\L}^{n}(D(\L)),  \Omega_{\L}^{n}(D(\L)) \right) 
    &\ \cong\  \Ext_\L^{i+p}\left( \Omega_{\L}^{n}(D(\L)),  \Omega_{\L}^{n}(D(\L)) \right) \\
    &\ \ \,  \vdots \\
    &\ \cong\   \Ext_\L^{i+m}\!\left( \Omega_{\L}^{n}(D(\L)),  \Omega_{\L}^{n}(D(\L)) \right) \\
    &\ \cong\   \Ext_\L^{i+m-n}\!\left( \Omega_{\L}^{n}(D(\L)), D(\L) \right) \\
    &\ = 0.
\end{align*}
Hence the fact that $\L$ satisfies (ARC) implies that $\injdim_{\L^{\rm op}}\L$ $=$ $\projdim_{\L}D(\L)$ $\leq n$, so that
$\L$ is Gorenstein by Proposition \ref{claim_19}.
The last statement follows from Proposition \ref{claim_44}.
\end{proof}

Now, we focus on Gorenstein projective modules over an eventually periodic algebra. 
Recall from \cite{Saito_2021} that a triangulated category $\mathcal{T}$ with shift functor $\Sigma$ is {\it periodic} if  $\Sigma^{m} \cong \mathrm{Id}_{\mathcal{T}}$ for some $m>0$. 
The smallest such $m$ is called the {\it period} of $\mathcal{T}$. 
We then have the following observation.

\begin{proposition} \label{claim_41}   
Let  $\L$ be an $(n,p)$-eventually periodic algebra.
Then the following statements hold.
\begin{enumerate}
  \setlength{\itemsep}{1.5mm}  
    \item $\L\sGProj$ and $\L\sGproj$ are periodic of period dividing $p$.
    \item $\L\GProj = p\mbox{-}\L\SGProj$ and   $\L\Gproj = p\mbox{-}\L\SGproj$.
\end{enumerate}
\end{proposition}

\begin{proof}
Let $\Sigma$ denote the shift functor on $\L\sGProj$. 
Since $\Omega_\L^{n+p} \cong \Omega_\L^{n}$ as endofunctors on $\L\sMod$ by Lemma \ref{claim_42}, the fact that $\Sigma^{-1} = \Omega_\L$ implies that $\Sigma^{-n-p} \cong\Sigma^{-n}$ and hence $\Sigma^{p} \cong \mathrm{Id}$.
Since $\Sigma$ restricts to the shift functor on $\L\sGproj$, we conclude that $\L\sGProj$ and $\L\sGproj$ are both periodic.
Now, (2) immediately follows from (1).
\end{proof}

\begin{remark}
{\rm
The same statements as Proposition \ref{claim_41}   hold for $\Lop\GProj$ and $\Lop\Gproj$. We leave it to the reader to state and show the analogous result.
}\end{remark}

We end this section by showing that eventually periodic algebras are both left and right weakly Gorenstein. 
Although this is a consequence of Proposition \ref{claim_28} and \cite[Theorem C]{Christensen-Holm_2010}, we give another proof. 
Recall that an algebra $\L$ is  {\it left weakly Gorenstein} if  $\L\Gproj ={}^{\perp}\L$, where ${}^{\perp}\L$ is the full subcategory of $\L\mod$ given by 
\[{}^{\perp}\L := \left\{ M \in \L\mod \mid  \Ext_{\L}^{i}(M, \L) = 0 \mbox{ for all } i > 0 \right \}. \]
Also, the algebra $\L$ is called {\it right weakly Gorenstein} if $\Lop$ is left weakly Gorenstein. 
See \cite{Marczinzik_2019,Ringel-Zhang_2020} for more details.
Also, thanks to Chen \cite[page 16]{X-WChen_2017}, we have the following equality for any algebra $\L$:
\begin{align} \label{eq_5}
    {}^{\perp}(\L\Proj) = \left\{ M \in \L\Mod \mid  \Ext_{\L}^{i}(M, \L) = 0 \mbox{ for all } i > 0 \right\}.
\end{align}

\begin{proposition} \label{claim_40}  
Let $\L$ be an eventually periodic algebra.
Then the following statements hold.
\begin{enumerate}
  \setlength{\itemsep}{1mm}  
    \item  $\L\GProj = {}^{\perp}(\L\Proj)$ and $\Lop\GProj = {}^{\perp}(\Lop\Proj)$.
    \item  $\L$ is both left and right weakly Gorenstein. 
\end{enumerate}  
\end{proposition}
\begin{proof}
Assume that $\L$ is $(n,p)$-eventually periodic. We only prove that $\L\GProj = {}^{\perp}(\L\Proj)$ and that $\L$ is left weakly Gorenstein; the proof for the others is similar.
For the former, 
it suffices to show the inclusion $(\supseteq)$.
For any $M \in {}^{\perp}(\L\Proj)$, its $n$-th syzygy $\Omega_{\L}^{n}(M)$ is $p$-strongly Gorenstein projective because $\Omega_{\L}^{n+p}(M) \cong \Omega_{\L}^{n}(M)$ in $\L\sMod$ by Lemma \ref{claim_42}, and
\begin{align*}
    \Ext_\L^{i}\!\left(\Omega_{\L}^{n}(M), \L\right) \cong \Ext_\L^{i+n}(M, \L) =0
\end{align*}
for all $i>0$.
This implies that $\Gpd_\L M \leq n < \infty$.
But, we get $\Gpd_\L M = 0$ because $M$ is in ${}^{\perp}(\L\Proj)$.
We have thus proved that $\L\GProj = {}^{\perp}(\L\Proj)$ as claimed.
On the other hand, the latter follows from the following equality
\begin{align*}
    \L\Gproj  = \L\GProj \cap \L\mod = {}^{\perp}(\L\Proj) \cap \L\mod = {}^{\perp}\L,
\end{align*}
where the last equality is obtained from the formula (\ref{eq_5}).
\end{proof}



\section*{Acknowledgments} 
I would like to thank Takahiro Honma and Ryotaro Koshio for helpful conversations.
I am also grateful to Ryo Takahashi. I learned the observation, presented in Introduction, regarding eventually periodic modules over commutative Noetherian local rings from him.
Finally, I would like to thank the referee for kindly pointing out errors in my English and providing suggestions, which improved this paper.


\bibliographystyle{plain}

\bibliography{ref}

\end{document}